%% file: main.tex
\documentclass[11pt]{article}

\usepackage{amsfonts}
\usepackage{amsmath}
\usepackage{amsthm}
\usepackage{graphicx}
\usepackage{enumerate}

\newtheorem{theorem}{Theorem}[section]

\newtheorem{example}[theorem]{Example}
\newtheorem{definition}[theorem]{Definition}

\newtheorem{proposition}[theorem]{Proposition}

\newtheorem{corollary}[theorem]{Corollary}

\newtheorem{conjecture}[theorem]{Conjecture}

\newcommand{\inv}{inv}
\newcommand{\maj}{maj}
\newcommand{\tooth}{tooth}
\newcommand{\istat}{istat}

\title{Pattern Avoiding Linear Extensions of Rectangular Posets}

\author{David Anderson \\ Carleton College \\ Department of Mathematics and Statistics \\ 1 North College Street \\ Northfield, MN 55057\\ davidjosepha@gmail.com\\ \\ Eric S. Egge\footnote{Corresponding author.} \\ Carleton College \\ Department of Mathematics and Statistics \\ 1 North College Street \\ Northfield, MN 55057 \\ eegge@carleton.edu\\  \\ Manda Riehl \\ University of Wisconsin-Eau Claire \\ Department of Mathematics \\ Hibbard Hall 512 \\ Eau Claire, WI 54701 \\ riehlar@uwec.edu\\ \\ Lucas Ryan \\ Carleton College \\ Department of Mathematics and Statistics \\ 1 North College Street \\ Northfield, MN 55057\\ lucas\_r23@yahoo.com \\ \\ Ruth Steinke \\ Carleton College \\ Department of Mathematics and Statistics \\ 1 North College Street \\ Northfield, MN 55057\\ ruth.steinke@gmail.com\\ \\ Yuriko Vaughan \\ Carleton College \\ Department of Mathematics and Statistics \\ 1 North College Street \\ Northfield, MN 55057\\ ysvaughan@gmail.com}

\begin{document}

\maketitle

\begin{abstract}
Inspired by Yakoubov's 2015 investigation of pattern avoiding linear extensions of the posets called combs, we study pattern avoiding linear extensions of rectangular posets.
These linear extensions are closely related to standard tableaux.
For positive integers $s$ and $t$ we consider two natural rectangular partial orders on $\{1,2,\ldots,s t\}$, which we call the NE rectangular order and the EN rectangular order.
First we enumerate linear extensions of both rectangular orders avoiding most sets of patterns of length three.
Then we use both a generating tree and a bijection to show that the linear extensions of the EN rectangular order which avoid 1243 are counted by the Fuss-Catalan numbers.
Next we use the transfer matrix method to enumerate linear extensions of the EN rectangular order which avoid 2143.
Finally, we open an investigation of the distribution of the inversion number on pattern avoiding linear extensions.

\medskip

\end{abstract}

\section{Introduction}

In \cite{combs} Yakoubov introduced an extensive new family of permutation enumeration problems.
To state the most general of these problems, suppose $n$ is a positive integer, $\Uparrow$ is a partial order on $[n]$, and $\sigma_1,\ldots,\sigma_k$ are permutations (see Section \ref{sec:bn} for definitions and notation).
Then Yakoubov's problem is to determine how many permutations of $[n]$ avoid $\sigma_1,\ldots,\sigma_k$ and also have the property that if $\pi(i) \Uparrow \pi(j)$ then $i < j$.
In other words, how many linear extensions of the poset $([n],\Uparrow)$ avoid $\sigma_1,\ldots,\sigma_k$?

As Yakoubov points out, this problem is hopelessly general without some additional information about $\Uparrow$.
For example, if $\Uparrow$ is empty (meaning no two elements of $[n]$ are related by $\Uparrow$) then Yakoubov's question reduces to the problem of enumerating the permutations avoiding a given set of patterns, a problem about which much has been written over the past thirty years, and about which much more is still unknown.
On the other hand, as Yakoubov also illustrates, for some particular families of partial orders we can make significant progress for a variety of short forbidden patterns.
In particular, Yakoubov obtains simple closed formulas for the number of linear extensions of posets she calls combs (see Figure \ref{fig:comb} for two typical examples) which avoid various sets of patterns of length three.

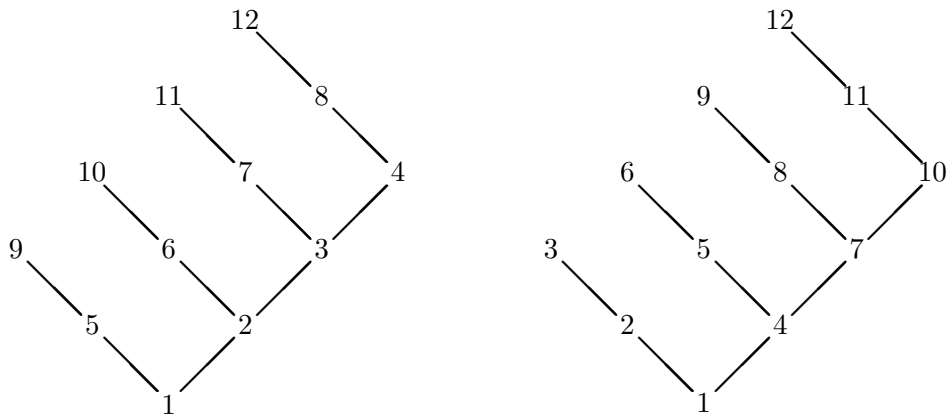
\begin{figure}[ht]
\setlength{\unitlength}{.4in}
\centering
\begin{picture}(12,5.2)
\thicklines
\multiput(0.15,1.85)(1,1){4}{\line(1,-1){.7}}
\multiput(1.15,.85)(1,1){4}{\line(1,-1){.7}}
\multiput(2.15,0.15)(1,1){3}{\line(1,1){.7}}
\put(1.9,-.15){$1$}
\put(2.9,.9){$2$}
\put(3.9,1.9){$3$}
\put(4.9,2.9){$4$}
\put(.9,.9){$5$}
\put(1.9,1.9){$6$}
\put(2.9,2.9){$7$}
\put(3.9,3.9){$8$}
\put(-.1,1.9){$9$}
\put(.8,2.9){$10$}
\put(1.8,3.9){$11$}
\put(2.8,4.9){$12$}
\multiput(7.15,1.85)(1,1){4}{\line(1,-1){.7}}
\multiput(8.15,.85)(1,1){4}{\line(1,-1){.7}}
\multiput(9.15,0.15)(1,1){3}{\line(1,1){.7}}
\put(8.9,-.12){$1$}
\put(9.9,.9){$4$}
\put(10.9,1.9){$7$}
\put(11.8,2.9){$10$}
\put(7.9,.9){$2$}
\put(8.9,1.9){$5$}
\put(9.9,2.9){$8$}
\put(10.8,3.9){$11$}
\put(6.9,1.9){$3$}
\put(7.9,2.9){$6$}
\put(8.9,3.9){$9$}
\put(9.8,4.9){$12$}
\end{picture}
\caption{The Hasse diagrams of a comb of type $\alpha$ (left) and $\beta$ (right).}
\label{fig:comb}
\end{figure}

In this paper we extend Yakoubov's investigation by studying pattern avoiding linear extensions of rectangular posets (see Figures \ref{fig:03rect}
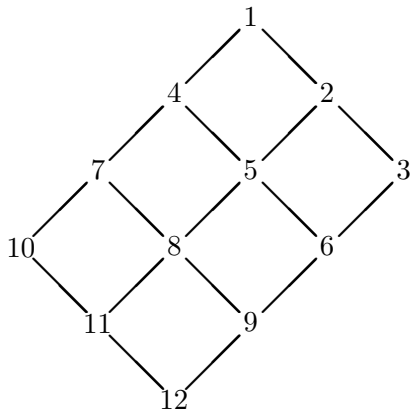
\begin{figure}[ht]
\centering
\setlength{\unitlength}{.4in}
\begin{picture}(5,5.2)
\thicklines
\multiput(0.15,1.85)(1,1){4}{\line(1,-1){.7}}
\multiput(1.15,.85)(1,1){4}{\line(1,-1){.7}}
\multiput(2.15,0.15)(1,1){3}{\line(1,1){.7}}
\multiput(1.15,1.15)(1,1){3}{\line(1,1){.7}}
\multiput(.15,2.15)(1,1){3}{\line(1,1){.7}}
\put(1.8,-.12){$12$}
\put(2.9,.9){$9$}
\put(3.9,1.9){$6$}
\put(4.9,2.9){$3$}
\put(.8,.88){$11$}
\put(1.9,1.9){$8$}
\put(2.9,2.9){$5$}
\put(3.9,3.9){$2$}
\put(-.2,1.88){$10$}
\put(.9,2.9){$7$}
\put(1.9,3.9){$4$}
\put(2.9,4.9){$1$}
\end{picture}
\caption{The Hasse diagram of the rectangular poset $NE_{4,3}$.}
\label{fig:03rect}
\end{figure}
and \ref{fig:30rect}
\begin{figure}[ht]
\setlength{\unitlength}{.4in}
\centering
\begin{picture}(5,5)
\thicklines
\multiput(0.15,1.85)(1,1){4}{\line(1,-1){.7}}
\multiput(1.15,.85)(1,1){4}{\line(1,-1){.7}}
\multiput(2.15,0.15)(1,1){3}{\line(1,1){.7}}
\multiput(1.15,1.15)(1,1){3}{\line(1,1){.7}}
\multiput(.15,2.15)(1,1){3}{\line(1,1){.7}}
\put(1.86,-.1){$10$}
\put(2.9,.9){$7$}
\put(3.9,1.9){$4$}
\put(4.9,2.9){$1$}
\put(.86,.9){$11$}
\put(1.9,1.9){$8$}
\put(2.9,2.9){$5$}
\put(3.9,3.9){$2$}
\put(-.14,1.9){$12$}
\put(.9,2.9){$9$}
\put(1.9,3.9){$6$}
\put(2.9,4.9){$3$}
\end{picture}
\caption{The Hasse diagram of the rectangular poset $EN_{4,3}$.}
\label{fig:30rect}
\end{figure}
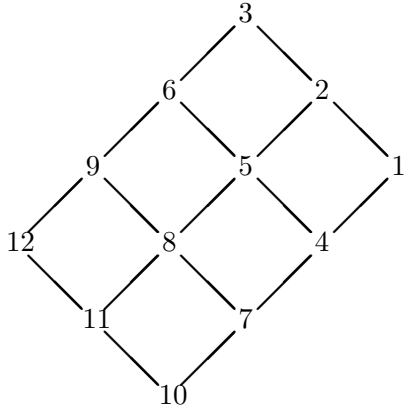
for two typical examples).
In Section \ref{sec:bn} we give some background on pattern avoidance and partially ordered sets, we define the particular posets we plan to study, and we prove some preliminary results to reduce the scope of our problem.
In Section \ref{sec:lengththree} we consider linear extensions avoiding sets of patterns of length three, finding simple closed formulas in most cases.
In Section \ref{sec:gentrees} we use a generating tree to show that the Fuss-Catalan numbers enumerate the linear extensions of a particular poset which avoid 1243.
We then give a natural bijection between this set of linear extensions and the set of Fuss-Catalan paths.
In Section \ref{sec:transfermatrices} we use the transfer matrix method to enumerate a set of linear extensions avoiding 2143.
In Section \ref{sec:statistics} we extend Yakoubov's work in another direction, studying the distribution of the inversion number on pattern avoiding linear extensions of our rectangular posets.
Finally, in Section \ref{sec:opfd} we describe future directions for this research along with some open problems.

\section{Background and Notation}
\label{sec:bn}

Our problem has two main ingredients:  pattern avoiding permutations and linear extensions of rectangular partially ordered sets.
To describe pattern avoiding permutations, let $n$ be a positive integer, let $[n]$ denote the set $\{1,2,\ldots,n\}$, and let $S_n$  denote the set of permutations of $[n]$, written in one-line notation.
For any permutation $\pi$, we write $|\pi|$ to denote the length of $\pi$, so $|\pi| = n$ is equivalent to $\pi \in S_n$.
We say a permutation $\pi$ {\em contains} a permutation $\sigma$ whenever $\pi$ has a subsequence with the same relative order as $\sigma$, and we say $\pi$ {\em avoids} $\sigma$ whenever $\pi$ has no such subsequence.
For example, $786549312$ has subsequence $8491$, so it contains $3241$.
However, $786549312$ has no subsequence with the same relative order as $132$, so $786549312$ avoids $132$.
In this context we sometimes call $\sigma$ a {\em forbidden pattern} and we sometimes call $\pi$ a {\em pattern avoiding permutation}.
Pattern avoiding permutations have received a considerable amount of attention over the past thirty years;  for more information, see Kitaev's encyclopedic book \cite{Kitaev} and the references he cites.

Our language and notation for partially ordered sets, their Hasse diagrams, and their linear extensions will follow \cite[Chap.~3]{StanleyVol1}, but for convenience we summarize these ideas here.
Recall that a {\em partial ordering} $\Uparrow$ of a set $X$ is a reflexive, antisymmetric, and transitive binary relation on $X$.
A {\em partially ordered set} (or {\em poset}, for short) $(X,\Uparrow)$ is a set $X$ together with a partial ordering $\Uparrow$ of $X$.
If $a,b \in X$ are distinct elements of a poset $(X,\Uparrow)$, then we say $a$ {\em covers} $b$ whenever $b \Uparrow a$ and there is no element $c \in X$ which is distinct from $a$ and $b$ and for which $b \Uparrow c \Uparrow a$.
A {\em Hasse diagram} of a poset $(X,\Uparrow)$ is a drawing of the graph whose vertices are the elements of $X$, in which two vertices are connected by an edge whenever one vertex covers the other in the partial ordering.
Typically we make this partial ordering apparent in the diagram by drawing $a$ above $b$ whenever $b \Uparrow a$.
Because Hasse diagrams are often the clearest way to describe the relations in a poset, we sometimes use them to define our posets.

For any poset $P = (X,\Uparrow)$, a {\em linear extension} of $P$ is a total ordering $\prec$ of $X$ which is consistent with $\Uparrow$.
In other words, if $a \Uparrow b$ then $a \prec b$.
When $X = [n]$ for some positive integer $n$, these linear extensions are naturally associated with permutations in $S_n$.
Specifically, the permutation $\pi$ associated with $\prec$ is the permutation with $\pi(1) \prec \pi(2) \prec \cdots \prec \pi(n)$.
For example, if $P$ is the poset whose Hasse diagram is given in Figure \ref{fig:Xposet}, 
\begin{figure}[h]
\centering
\setlength{\unitlength}{.4in}
\begin{picture}(2,2)
\thicklines
\multiput(.15,.15)(1,1){2}{\line(1,1){.7}}
\multiput(1.85,.15)(-1,1){2}{\line(-1,1){.7}}
\put(-0.1,-0.1){$5$}
\put(1.9,-.1){$2$}
\put(.9,.9){$4$}
\put(-.1,1.9){$3$}
\put(1.9,1.9){$1$}
\end{picture}
\caption{The Hasse diagram of a poset on $[5]$.}
\label{fig:Xposet}
\end{figure}
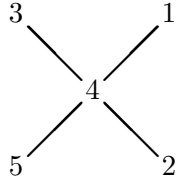
then $P$ has four linear extensions:  $25413$, $25431$, $52413$, and $52431$.

Many of our posets will have the property that their Hasse diagrams can be represented as tilted rectangles, as in Figure \ref{fig:30rect};  we call such a poset {\em rectangular}.
Following Yakoubov, for a given rectangular representation of a rectangular poset, we call the sequences of elements in the diagonals from lower left to upper right the {\em spines}, and we write $s$ to denote the number of elements in each spine.
Similarly, we call the sequences of elements in the diagonals from lower right to upper left the {\em teeth}, and we write $t$ to denote the number of elements in each tooth.
We number the spines from upper left to lower right, and the teeth from upper right to lower left.
In Figure \ref{fig:30rect} we have $s = 4$, $t = 3$, the first spine contains $3$, $6$, $9$, and $12$, and the third tooth contains $7$, $8$, and $9$.

Given positive integers $s$ and $t$, there are eight natural rectangular partial orderings of $[st]$ with spine length $s$ and tooth length $t$.
For example, in one natural rectangular partial ordering the first spine is $1,2,\ldots,s$ from first tooth to last, the second spine is $s+1,s+2,\ldots, 2s$ from first tooth to last, and in general the $j$th spine is $(j-1)s +1, (j-1)s+2, \ldots, js$ from first tooth to last.
In another natural rectangular partial ordering the $s$th tooth is $t, t-1, \ldots, 1$ from first spine to last, the $(s-1)$th tooth is $2t, 2t-1, \ldots, t+1$ from first spine to last, and in general the $(s-j)$th tooth is $(j+1)t, (j+1)t-1, \ldots, jt + 1$ from first spine to last.
In each of these eight natural rectangular partial orderings one spine or tooth is $1,2,\ldots$ in order, so we name each ordering according to the corners of the Hasse diagram (North, East, South, or West) at which this spine or tooth begins and ends.
For example, in Figure \ref{fig:30rect} we have the Hasse diagram for $EN_{4,3}$, and in Figure \ref{fig:sw24}
\begin{figure}[ht]
\setlength{\unitlength}{.35in}
\centering
\begin{picture}(4,4)
\thicklines
\multiput(3.15,0.15)(-1,1){4}{\line(1,1){.7}}
\multiput(2.85,.15)(1,1){2}{\multiput(0,0)(-1,1){3}{\line(-1,1){.7}}}
\put(2.9,-.1){$1$}
\put(1.9,.9){$2$}
\put(.9,1.9){$3$}
\put(-.1,2.9){$4$}
\put(3.9,.9){$5$}
\put(2.9,1.9){$6$}
\put(1.9,2.9){$7$}
\put(.9,3.9){$8$}
\end{picture}
\caption{The Hasse diagram of the rectangular poset $SW_{2,4}$.}
\label{fig:sw24}
\end{figure}
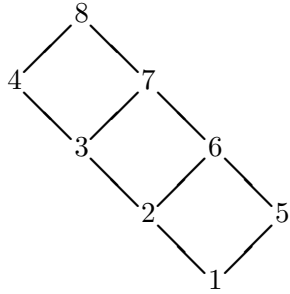
we have the Hasse diagram for $SW_{2,4}$.
Finally, for any rectangular poset $P = ([st],\Uparrow)$ and any permutations $\sigma_1,\ldots, \sigma_n$, we write $P(\sigma_1,\ldots,\sigma_n)$ to denote the set of linear extensions of $P$ which avoid each of $\sigma_1,\ldots,\sigma_n$.

Although there are eight natural rectangular partial orderings on $[st]$, when counting their pattern avoiding linear extensions we need only concern ourselves with two of them.
For instance, by reflecting their Hasse diagrams over a vertical line, we can see that $EN_{s,t} = WN_{t,s}$, $NE_{s,t} = NW_{t,s}$, $ES_{s,t} = WS_{t,s}$, and $SE_{s,t} = SW_{t,s}$.
Similarly, reflecting a Hasse diagram over a horizontal line has the effect of reversing the linear extensions of the associated poset.
Therefore, for any permutations $\sigma_1,\ldots, \sigma_n$ we have $|ES_{s,t}(\sigma_1,\ldots,\sigma_n)| = |EN_{t,s}(\sigma_1^r,\ldots,\sigma_n^r)|$, $|SE_{s,t}(\sigma_1,\ldots,\sigma_n)| = |NE_{t,s}(\sigma_1^r,\ldots,\sigma_n^r)|$, $|WS_{s,t}(\sigma_1,\ldots,\sigma_n)| = |WN_{t,s}(\sigma_1^r,\ldots,\sigma_n^r)|$, and finally $|SW_{s,t}(\sigma_1,\ldots,\sigma_n)| = |NW_{t,s}(\sigma_1^r,\ldots,\sigma_n^r)|$, where $\sigma_j^r$ is the reverse of $\sigma_j$.
Combining these observations, we see we can restrict our attention to $EN_{s,t}(\sigma_1,\ldots,\sigma_n)$ and $NE_{s,t}(\sigma_1,\ldots,\sigma_n)$.

In addition to reducing the collection of posets we need to consider, we can also reduce the collection of forbidden patterns we need to consider.
In particular, if we write $\sigma^{rc}$ to denote the reverse complement of a permutation $\sigma$, then we have the following result.

\begin{proposition}
\label{prop:rc}
For any rectangular poset $P$ and any permutations $\sigma_1,\ldots,\sigma_n$, we have 
$$|P(\sigma_1,\ldots,\sigma_n)| = |P(\sigma_1^{rc},\ldots,\sigma_n^{rc})|.$$
\end{proposition}
\begin{proof}
We prove the result for $EN_{s,t}$;  the other cases are similar.

Since a permutation $\pi$ avoids a permutation $\sigma$ if and only if $\pi^{rc}$ avoids $\sigma^{rc}$, and $rc$ is invertible, it's sufficient to show that the set of linear extensions of $EN_{s,t}$ is closed under $rc$.
To do this, we describe $rc$ in terms of bijections on sets of linear extensions.

We have observed that reflecting the Hasse diagram of $EN_{s,t}$ over a vertical line induces a bijection between the set of linear extensions of $EN_{s,t}$ and the set of linear extensions of $WN_{t,s}$.
In fact, this bijection is the identity map.
Similarly, reflection over a horizontal line induces a bijection between the set of linear extensions of $WN_{t,s}$ and the set of linear extensions of $WS_{s,t}$.
In fact, this bijection simply reverses each linear extension.
If we now replace each $j \in [st]$ with its complement $st+1-j$, then we have a bijection from the set of linear extensions of $WS_{s,t}$ to the set of linear extensions of $EN_{s,t}$.
The composition of these three maps is $rc$, so the set of linear extensions of $EN_{s,t}$ is closed under $rc$, as desired.
\end{proof}

We now have all of the background and notation we need to begin counting pattern avoiding linear extensions of $EN_{s,t}$ and $NE_{s,t}$.
Before we do this, it's natural to ask what happens when there are no patterns to avoid.
That is, how many linear extensions of $EN_{s,t}$ and $NE_{s,t}$ are there?

To address this question, first note that the answer depends only on the shape of the Hasse diagram, and not on the names of the elements of the posets, so $EN_{s,t}$ and $NE_{s,t}$ have the same number of linear extensions.
Furthermore, the linear extensions of $EN_{s,t}$ are in bijection with the labelings of the vertices of its Hasse diagram with $1,2,\ldots,st$ such that if $x \Uparrow y$ for vertices $x$ and $y$ then the label on $x$ is less than the label on $y$.
Now recall that a {\em standard tableaux of shape $\underbrace{t,\cdots, t}_s = t^s$} is a filling of a $s \times t$ rectangle with $1,2,\ldots,st$ in which the entries in each row are strictly increasing from left to right, the entries in each column are strictly increasing from top to bottom, and each of $1,2,\ldots,st$ appears exactly once.
Suppose we have a linear extension $\pi$ of $EN_{s,t}$.
Number each vertex of the Hasse diagram for $EN_{s,t}$ with its position in $\pi$, rotate the diagram through $3\pi/4$ radians clockwise, and enclose each vertex in a square.
When $\pi$ is the linear extension of $EN_{4,3}$ given by $10\ 7\ 11\ 4\ 1\ 8\ 12\ 5\ 2\ 9\ 6\ 3$, the resulting object is the standard tableaux in Figure \ref{fig:stdtabex}.
\begin{figure}[ht]
\centering
\begin{tabular}{|c|c|c|}
\hline
$1$ & $3$ & $7$ \\
\hline
$2$ & $6$ & $10$ \\
\hline
$4$ & $8$ & $11$ \\
\hline
$5$ & $9$ & $12$ \\
\hline
\end{tabular}
\caption{The standard tableaux associated with $10\ 7\ 11\ 4\ 1\ 8\ 12\ 5\ 2\ 9\ 6\ 3$.}
\label{fig:stdtabex}
\end{figure}
In general, the resulting object is a standard tableaux of shape $t^s$, and this procedure is a bijection between the set of linear extensions of $EN_{s,t}$ and the set of standard tableaux of shape $t^s$.
Therefore, the classical hook length formula \cite{wilfhooklength} \cite[Sec.~3.10]{sagan} gives us the following result. 

\begin{proposition}
\label{prop:totalcount}
For any positive integers $s$ and $t$, the number of linear extensions of $EN_{s,t}$ (or $NE_{s,t}$) is ${\displaystyle (st)! \prod_{j=1}^t \frac{(s+t-j)!}{(j-1)!}}$.
\end{proposition}

Since the Catalan numbers are so ubiquitous, it's worth noting their appearance in a special case of Proposition \ref{prop:totalcount}.

\begin{corollary}
\label{cor:linextCat}
For any positive integer $n$, the number of linear extensions of any of $EN_{n,2}$, $NE_{n,2}$ $EN_{2,n}$, or $EN_{2,n}$ is the Catalan number ${\displaystyle C_n = \frac{1}{n+1} \binom{2n}{n}}$.
\end{corollary}

\section{Avoiding Patterns of Length Three}
\label{sec:lengththree}

\input{section3}

\section{$EN(1243)$, Generating Trees, and Fuss-Catalan Paths}
\label{sec:gentrees}

\input{section4}

\section{Catalan Zippers and $EN(2143)$}
\label{sec:transfermatrices}

\input{section5}

\section{Inversions and Pattern Avoiding Linear Extensions}
\label{sec:statistics}

\input{section6}

\section{Open Problems and Future Directions}
\label{sec:opfd}

\input{section7}

\bibliographystyle{alpha}
\bibliography{references}

\end{document}

%% file: section3.tex
In this section we enumerate the linear extensions of $EN_{s,t}$ and $NE_{s,t}$ avoiding various sets of patterns of length three. 
We begin with linear extensions of $EN_{s,t}$.

\begin{theorem}
\label{thm:EN132213}
For all $s \ge 1$ and all $t \ge 1$ we have $|EN_{s,t}(213)| = 1$.
\end{theorem}
\begin{proof}
Notice that if any element of the $i$th tooth precedes an element of the $(i+1)$th tooth in a linear extension $\pi$ of $NE_{s,t}$, then $\pi$ contains 213. 
Therefore, only the linear extension $st-t+1, st-t+2,\ldots, st, st-2t+1,\ldots, st-t, \ldots, 1, 2,\dots, t$ can be in $EN_{s,t}(213)$. 
This permutation avoids 213, and the result follows. 
\end{proof}

Our proof of Theorem \ref{thm:EN132213} amounts to showing that $EN_{s,t}(213)$ is the set of linear extensions of the poset we obtain by adding certain covering relations to $EN_{s,t}$. 
In particular, in this new poset we require that $1+(j-1)t$ (the lower right element of the $j$th tooth) covers $(j+1)t$ (the upper left element of the $(j+1)$th tooth).
See Figure \ref{fig:213saw}
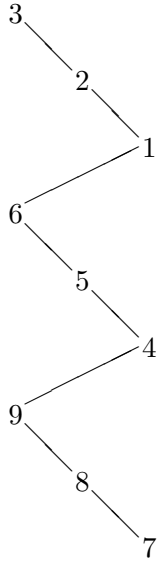
\begin{figure}[ht]
\centering
\setlength{\unitlength}{.35in}
\begin{picture}(2,8)
\put(1.85,.15){\line(-1,1){.7}}
\put(.85,1.15){\line(-1,1){.7}}
\put(1.85,3.15){\line(-1,1){.7}}
\put(.85,4.15){\line(-1,1){.7}}
\put(1.85,6.15){\line(-1,1){.7}}
\put(.85,7.15){\line(-1,1){.7}}
\put(0.15,2.15){\line(2,1){1.7}}
\put(0.15,5.15){\line(2,1){1.7}}
\put(1.9,-.15){$7$}
\put(.9,.85){$8$}
\put(-.1,1.85){$9$}
\put(1.9,2.85){$4$}
\put(.9,3.85){$5$}
\put(-.1,4.85){$6$}
\put(1.9,5.85){$1$}
\put(.9,6.85){$2$}
\put(-.1,7.85){$3$}
\end{picture}
\caption{The modified poset related to $EN_{s,t}(213)$ when $s = t = 3$.}
\label{fig:213saw}
\end{figure}
for the Hasse diagram of this new poset when $s = t = 3$.
This method of exchanging additional covering relations for pattern avoidance restrictions will prove useful later.

\begin{theorem}
\label{thm:EN231312}
For all $s \ge 2$ and all $t \ge 2$,
\begin{enumerate}[{\upshape (i)}]
\item{$|EN_{1,t}(231)| = 1$;}
\item{$|EN_{s,1}(231)| = 1$;}
\item{$|EN_{s,t}(231)|= 0$.}
\end{enumerate}
\end{theorem}
\begin{proof}
Notice that if any two elements of the $i$th tooth precede any element in the $(i-1)$th tooth in a linear extension $\pi$, then $\pi$ will contain $231$. 
Thus, it is not possible for any linear extension of a poset with more than one spine or more than one tooth to avoid $213$. 
Therefore, only the linear extensions $1, 2, ..., t-1, t$ of $EN_{1,t}$ and $s, s-1, s-2, ..., 2, 1$ of $EN_{s,1}$ avoid $213$. 
\end{proof}
 
\begin{theorem}
For all $s \ge 3$ and all $t \ge 1$,
\label{thm:EN, 321}
\begin{enumerate}[{\upshape (i)}] 
\item{$|EN_{1,t}(321)| = 1$;} 
\item{$|EN_{2,t}(321)|=C_t$;} 
\item{$|EN_{s,t}(321)| = 0$.}
\end{enumerate}
\end{theorem}
\begin{proof}
When $s=1$, the poset has one linear extension, namely $1, 2, ..., t-1, t$, and this linear extension avoids 321. 

When $s=2$, every linear extension in $EN_{2,t}(321)$ is a union of two increasing subsequences of length $t$, corresponding to the teeth of $EN_{2,t}$. 
Each such permutation avoids 321, and the result follows from Corollary \ref{cor:linextCat}. 

When $s\geq 3$, every linear extension of $EN_{s,t}$ contains the subsequence $2t+1, t+1, 1$, so none of these linear extensions avoid 321.  
\end{proof}
  
\begin{theorem}
\label{thm:EN, 123}
For all $s \ge 1$ and all $t \ge 3$,
\begin{enumerate} [{\upshape (i)}] 
\item{$|EN_{s,1}(123)|=1$;} 
\item{$|EN_{s,2}(123)|=C_s$;}
\item{$|EN_{s,t}(123)|=0$.}
\end{enumerate}
\end{theorem}
\begin{proof}
The proof is similar to that of Theorem \ref{thm:EN, 321}.
\end{proof}    

In view of Proposition \ref{prop:rc}, we have now found $|EN_{s,t}(\sigma)|$ for all $\sigma \in S_3$, and by combining our results we can easily find the number of linear extensions of $EN_{s,t}$ which avoid any list of patterns in $S_3$.
With this done, we turn our attention to linear extensions of $NE_{s,t}$ which avoid various patterns of length three.

\begin{theorem}
\label{thm:NE(213)=NE(132)}
For all $s \ge 1$ and all $t \ge 1$,
\begin{equation}
|NE_{s,t}(213)| = t^{s-1}.
\end{equation}
\end{theorem}
\begin{proof}
We argue by induction on $s$.
          
When $s=1$, the poset only has one linear extension, namely $st, st-1, st-2, \ldots, 2,1$, and this linear extension avoids 213. 
Thus $NE_{1,t}(213)=t^0=1$.
          
Now we will show that $|NE_{s+1,t}(213)| = t \cdot t^{s-1}$ by mapping every linear extension in $NE_{s,t}(213)$ to $t$ linear extensions in $NE_{s+1,t}(213)$. Let $S=a_{st},\dots,a_1$ be a linear extension in $NE_{s,t}(213)$. 
We can map $S$ to the linear extension in $NE_{s+1,t}(213)$ that we obtain by inserting $(s+1)t,\dots,st+1$ into $S$ (in this order) so that the resulting linear extension still avoids 213. We must add $(s+1)t$ before $a_{st}$ (because it is now the largest element in the poset). 
To ensure that we map to a linear extension that avoids 213, we must insert the rest of the numbers $b \in \{(s+1)t-1,\dots,st+1\}$ into two places: directly between $(s+1)t$ and $a_{st}$, or immediately after $a_{st}$. 
Otherwise, the subsequence $a_{st}, k, b$ (where $k\in S$ and $k < a_{st}$) creates a 213 pattern. 
Since $(s+1)t-1,\dots,st+1$ must appear greatest to least, there are exactly $t$ ways of inserting them into the two possible locations. 
\end{proof}

\begin{example}
We can insert $9$, $8$, and $7$ into $635421$, which is an element of $NE_{2,3}(213)$, in three ways to obtain linear extensions in $NE_{3,3}(213)$.
The resulting linear extensions are $\underline{987} 635421$, $\underline{98} 6 \underline{7}35421$, and $\underline{9} 6 \underline{87}35421$.
Here we have underlined the newly inserted entries.
\end{example}

Later, we will revisit this technique of constructing linear extensions of a larger poset by inserting elements into linear extensions of a smaller poset. 
However, as an immediate consequence of the form linear extensions avoiding 213 must take, as described in the above proof, we can enumerate the linear extensions of $NE_{s,t}$ which avoid particular sets of two patterns. 

\begin{corollary}
\label{cor:NE(123,213)=NE(123,132)}
For all $s \ge 1$ and all $t \ge 1$, 
$$|NE_{s,t}(213,123)| = t^{s-1}.$$
\end{corollary}
\begin{proof}
We claim that every linear extension in $NE_{s,t}(213)$ is also in $NE_{s,t}(123)$, arguing by induction on $s$. The base case is clear from the previous theorem, since the single linear extension of $NE_{1,t}$ is strictly decreasing and therefore avoids both 213 and 123. 

According to the algorithm described in the proof of Theorem \ref{thm:NE(213)=NE(132)} for inserting elements into a linear extension $a_{st}, a_{st-1}, ..., a_{2}, a_{1}$ to construct a linear extension in $NE_{s+1,t}(213)$, all linear extensions in $NE_{s+1,t}(213)$ must begin with a strictly decreasing sequence of the first $i$ ($1 \leq i \leq t$) elements from $(s+1)t, ..., st+1$, followed by $a_{st}$. Next, these linear extensions contain another strictly decreasing sequence of the remaining $t-i$ elements from $(s+1)t, ..., st+1$. Finally, each of these extended linear extensions ends with the original linear extension of the remaining $st-1$ elements of $NE_{s,t}$, which we now assume avoids both 123 and 213 as our inductive hypothesis. If a 123 pattern were to appear in the arrangement of these two sequences inserted around $a_{st}$, then it must contain $a_{st}$, since the inserted elements are in strictly decreasing order and therefore do not contain a 123 pattern. Moreover, $a_{st}$ would have to act as a ``1", since it is smaller than the other inserted elements.
However, the strictly decreasing subsequence of elements we inserted after $a_{st}$ are all greater than the final $st-1$ elements of the linear extension, so we cannot complete the 123 pattern with any increasing subsequence of length two. 
Hence, $NE_{s,t}(213,123)=NE_{s,t}(213)$.
\end{proof}

\begin{corollary}
\label{cor:NE(132,213)} For all $s \ge 1$ and all $t \ge 2$,
\begin{enumerate}[{\upshape (i)}]
\item{$|NE_{s,1}(213,132)|=1$;}
\item{$|NE_{s,t}(213,132)| = 2^{s-1}$.}
\end{enumerate}
\end{corollary}
\begin{proof}
We again use the insertion algorithm from the previous corollary. To avoid 132 as we insert our $t$ largest elements, we can place at most one element after $a_{st}$ and before the rest of the old linear extension of $NE_{s,t}$. Otherwise, $a_{st}, st+2, st+1$ form a 132 pattern. 
Hence, there are exactly two ways to build a linear extension in $NE_{s+1,t}(132,213)$ from a linear extension in $NE_{s,t}(132,213)$, and the result follows by induction. 
\end{proof}

\begin{theorem}
\label{thm:NE, 231&312}
For all $s \ge 1$ and all $t \ge 1$ we have $|NE_{s,t}(312)| = 1$.
\end{theorem}
\begin{proof}
Every linear extension of $NE_{s,t}$ begins with its largest element, $st$, so a linear extension can avoid 312 if and only if it is $st,st-1,\dots, 1$. 
This is, in fact, a linear extension of $NE_{s,t}$, so the result follows. 
\end{proof}

Combining the results in this section with Proposition \ref{prop:rc} allows us to find $|NE_{s,t}(\sigma_1,\ldots)|$ for any list of forbidden patterns of length three, with one notable exception:  so far we have said nothing about $|NE_{s,t}(123)|$.
In Table \ref{table:NEst123data}
\begin{table}[ht]
\centering
\begin{tabular}{c|cccccc}
$s \setminus t$ & 1 & 2 & 3 & 4 & 5 & 6 \\
\hline
1 & 1 & 1 & 1 & 1 & 1 & 1 \\
2 & 1 & 2 & 5 & 14 & 42 & 132 \\
3 & 1 & 5 & 33 & 234 & 1706 & 12618 \\
4 & 1 & 14 & 238 & 4146 & 72152 & 1246804 \\
5 & 1 & 42 & 1782 & 75187 & 3099106 \\
6 & 1 & 132 & 13593 & 1378668 \\
7 & 1 & 429 & 104756 & 25430445
\end{tabular}
\caption{$|NE_{s,t}(123)|$ for small $s$ and $t$.}
\label{table:NEst123data}
\end{table}
we have the values of $|NE_{s,t}(123)|$ for various small $s$ and $t$.
We leave it as an exercise for the reader to show that for all $s \ge 1$ and all $t \ge 1$ we have $|NE_{s,1}(123)| = |NE_{1,t}(123)| = 1$, $|NE_{s,2}(123) |= C_s$, and $|NE_{2,t}(123)| = C_t$.
It is an open problem to find $|NE_{s,t}(123)|$ for $s \ge 3$ and $t \ge 3$.

%% file: section4.tex
We now turn our attention to linear extensions of $EN_{s,t}$ avoiding 1243. 
In our proofs of Theorems \ref{thm:EN132213} and \ref{thm:EN231312}, we saw that the linear extensions of $EN_{s,t}$ which avoid 213 are the linear extensions of a poset we obtain by adding several new covering relations to $EN_{s,t}$, as are the linear extensions of $EN_{s,t}$ which avoid 231. 
As a first step in enumerating $EN_{s,t}(1243)$, we show that a similar result holds in this case.
However, our new posets, which we call {\em sawblade posets}, are somewhat more complicated.

\begin{definition}
A {\em sawblade poset}, denoted $SAW_{s,t}$, is the poset we obtain from $EN_{s,t}$ by adding covering relations such that $(j-1)t+2$ covers $(j+1)t$ for all $j$ with $1 \le j \le s-1$.
\end{definition}

In Figure \ref{fig:firstsaw43} 
\begin{figure}[ht]
\centering
\setlength{\unitlength}{.35in}
\begin{picture}(5,5)
\multiput(2.15,.15)(1,1){3}{\line(1,1){.7}}
\multiput(1.85,.15)(1,1){4}{\line(-1,1){.7}}
\multiput(.85,1.15)(1,1){4}{\line(-1,1){.7}}
\multiput(0.15,2)(1,1){3}{\line(1,0){1.7}}
\put(4.9,2.9){$1$}
\put(3.9,3.9){$2$}
\put(2.9,4.9){$3$}
\put(3.9,1.9){$4$}
\put(2.9,2.9){$5$}
\put(1.9,3.9){$6$}
\put(2.9,.9){$7$}
\put(1.9,1.9){$8$}
\put(.9,2.9){$9$}
\put(1.85,-.15){$10$}
\put(.85,.85){$11$}
\put(-.2,1.85){$12$}
\end{picture}
\caption{The sawblade poset $SAW_{4,3}$.}
\label{fig:firstsaw43}
\end{figure}
we have the Hasse diagram for $SAW_{4,3}$, drawn to suggest the teeth of a saw.
In Figure \ref{fig:SAW43}
\begin{figure}[ht]
\centering
\setlength{\unitlength}{.35in}
\begin{picture}(5,8)
\multiput(2.15,.15)(1,1){3}{\line(1,1){.7}}
\multiput(0.15,2.15)(0,2){3}{\line(1,1){.7}}
\put(1.85,.15){\line(-1,1){.7}}
\multiput(.85,1.15)(0,2){4}{\line(-1,1){.7}}
\put(2.85,1.15){\line(-1,1){1.7}}
\put(3.85,2.15){\line(-1,1){2.7}}
\put(4.85,3.15){\line(-1,1){3.7}}
\put(4.88,2.88){$1$}
\put(.88,6.88){$2$}
\put(-.12,7.88){$3$}
\put(3.88,1.88){$4$}
\put(.88,4.88){$5$}
\put(-.12,5.88){$6$}
\put(2.88,.88){$7$}
\put(.88,2.88){$8$}
\put(-.12,3.88){$9$}
\put(1.8,-.12){$10$}
\put(.8,.88){$11$}
\put(-.2,1.88){$12$}
\end{picture}
\caption{The sawblade poset $SAW_{4,3}$.}
\label{fig:SAW43}
\end{figure}
we have the Hasse diagram of $SAW_{4,3}$ drawn to respect the convention that if $a$ covers $b$ then $a$ is above $b$ in the diagram.

\begin{theorem}
\label{thm:sawblade}
For all $s \ge 0$ and all $t \ge 0$, a linear extension of $EN_{s,t}$ avoids 1243 if and only if it is a linear extension of $SAW_{s,t}$.
\end{theorem}
\begin{proof}
First, we will show that the property is necessary, that is, in any linear extension of $EN_{s,t}$ avoiding 1243, the entry $(j+1)t$ precedes the entry $(j-1)t+2$ for all $j$ with $1 \le j \le s-1$.
  
Suppose not, and $(j-1)t+2$ occurs before $(j+1)t$ for some particular $j$. 
Since $(j-1)t+1$ is below $(j-1)t+2$ on the same tooth, $(j-1)t+1$ precedes $(j-1)t+2$ in our linear extension. 
Similarly, since $jt$ is above $(j+1)t$ on the same spine, $(j+1)t$ precedes $jt$. 
Combining these observations, we see that $(j-1)t+1$, $(j-1)t+2$, $(j+1)t$, $jt$ is a subsequence with the same relative order as 1243.
  
To show that the given property is sufficient, suppose $abcd$ is a subsequence of a linear extension of $EN_{s,t}$ with the same relative order as 1243.
We make several observations about where $b$ and $c$ can occur in the Hasse diagram of $EN_{s,t}$.

First note that $b$ is not on the $t$th spine of $EN_{s,t}$:  if it were then all entries to its left in the linear extension would be greater than $b$, but $a$ is to the left of $b$ and less than $b$.

We now claim that $b$ and $c$ cannot be on the same tooth.
If $b$ and $c$ did occur on the same tooth, then every $z$ with $b < z < c$ would occur between $b$ and $c$ in the linear extension, because all such $z$ occur between $b$ and $c$ on their common tooth.
However, $b < d < c$ but $d$ occurs after $c$ in our linear extension.

We also observe that $c$ cannot be on a higher (that is, lower-numbered) tooth than $b$, because the elements of higher teeth are less than the elements of lower teeth, but $c > b$.

Now suppose $b$ is on the $j$th tooth.
Since $b$ is not on the $t$th spine, either $b = (j-1)t+2$ or $(j-1)t+2$ occurs before $b$.
On the other hand, $c$ is on a lower tooth than $b$, so either $c = (j+1)t$ or $(j+1)t$ occurs after $c$.
Combining these observations with the fact that $b$ precedes $c$ in our linear extension, we see that $(j-1)t+2$ precedes $(j+1)t$, as desired.
\end{proof}


Our first approach to enumerating the linear extensions of $SAW_{s,t}$ is modeled on our enumeration of $NE_{s,t}(213)$ in Theorem \ref{thm:NE(213)=NE(132)}. 
In particular, for each $s$, we look at how many linear extensions of $SAW_{s+1,t}$ we can obtain from a given linear extension of $SAW_{s,t}$ by adding a tooth. 
This number depends on the particular linear extension we choose, so we use a technical tool called a {\em generating tree} to keep track of it.

Recall from the work of West and others \cite{WestCatalan,West,TwoLabel} that a generating tree is a root with a particular label or labels together with a method for determining the children of any node from the label of that node. 
In our generating tree, each node at depth $s$ will correspond to a linear extension of $SAW_{s,t}$, and each node's children will correspond to the linear extensions of $SAW_{s+1,t}$ created by extending that linear extension. 
Therefore, our root will represent the empty linear extension of a $0 \times t$ poset. 
From here, we need a method for determining the children of a node, which will require assigning a label to each node.

  

  
\begin{definition}
For any positive integers $s$ and $t$, and any linear extension $\pi$ of $SAW_{s,t}$, the label of $\pi$, written $label(\pi)$, is the number $j$ such that 1 is in position $st-j$ of $\pi$.
\end{definition}

  
As we show next, the label of a given linear extension of $SAW_{s,t}$ determines the labels of its children.
  
\begin{theorem}
\label{thm:gen-tree-succession-rule}
For any positive integers $s$ and $t$, if $\pi$ is a linear extension of $SAW_{s,t}$ then the labels of the children of $\pi$ are
\begin{equation*}
t-1, t, t+1, \dots, label(\pi) + t-1.
\end{equation*}
\end{theorem}
\begin{proof}
We will construct each possible $\pi'$ from $\pi$. 
To start, we take $\pi$ and consider it as a partial linear extension of $SAW_{s+1,t}$, updating the values of elements by adding $t$, lacking only the elements $1, \dots, t$. 
The elements $2, \dots, t$ must come at the end of this linear extension by the covering relations of the poset. 
The only element whose position is not fixed is 1. 
We then construct a child $\pi'$ of $\pi$ for each possible position of 1, which can be placed at any point after $label(\pi)+t$ and before 2. 
These correspond to children with labels $t-1, t, t+1, \dots, label(\pi) + t-1$, respectively.
\end{proof}
  
Note that each linear extension $\pi$ of $SAW_{s,t}$ has $t + label(\pi)$ children.
  
\begin{corollary}
\label{cor:gen-tree-iso}
For any $s \ge 1$, the generating tree of linear extensions of $SAW_{s,t}$ is isomorphic to the generating tree given by
\begin{align*}
Root: \hspace{.5cm} &(0)\\
Rule: \hspace{.5cm} &(j) \rightarrow (t-1)(t)\dots(j+t-1)
\end{align*}
\end{corollary}
\begin{proof}
This is immediate from Theorem \ref{thm:gen-tree-succession-rule}.
\end{proof}

We call the generating tree in Corollary \ref{cor:gen-tree-iso} the {\em $t$-Fuss-Catalan} generating tree. 
Note that the 2-Fuss-Catalan generating tree is isomorphic to West's Catalan generating tree \cite[Ex.~4]{West}.
  

\begin{theorem}
\label{thm:gen-tree-func}
For $s,t \ge 1$, the number of nodes at depth $s$ in the $t$-Fuss-Catalan tree is the Fuss-Catalan number 
\begin{equation*}
\frac{1}{st+1} \binom{st+1}{s} = \frac{1}{(t-1)s+1} \binom{st}{s}.
\end{equation*}
\end{theorem}
\begin{proof}
Define $G(x,y)$ to be the generating function for the nodes of the generating tree given as follows: 
$$G(x,y) = \sum_\sigma x^{l(\sigma)} y^{p(\sigma)},$$ 
with $l(\sigma)$ denoting the depth of the node $\sigma$ in the tree (with the root at depth $0$), and $p(\sigma)$ indicating the label of $\sigma$. 
This generating function can equivalently be written in terms of $G_p(x)$, the coefficient of $y^p$. 
That is, we have 
$$G(x,y) = \sum_{p \ge 0} G_p(x) y^p.$$
Since each node other than the root is a child of another node, we have
\begin{align*}
G(x,y) &= x^0 y^0 + x \sum_{p \ge 0} G_p(x) (y^{t-1} + y^{t} + \dots + y^{p+t-1})\\
&= 1 + \frac{x y^{t-1}}{y-1} \left( y G(x,y) - G(x,1) \right).
  \end{align*}
Rearranging, we find 
$$G(x,y) \left(1 - \frac{x y^t}{y-1} \right) = 1 - \frac{x y^{t-1}}{y-1} G(x;1).$$
We now use the kernel method \cite{kernel} to find $G(x,1)$. 
In particular, we set 
$$y(x) = \sum_{s=0}^{\infty} \frac{1}{st+1} \binom{st+1}{s} x^s$$ 
and we use the fact \cite[Ex.~5 in Sec.~7.5]{Graham} that ${\displaystyle 1 - \frac{x y^{t}}{y-1} = 0}$ to find 
$$G(x,1) = \frac{y-1}{x y^{t-1}}.$$
Now the fact that $y-1 = x y^t$ means $G(x,1) = y(x)$, and the result follows.
\end{proof}

\begin{corollary}
\label{cor:1243 formula}
For all $s,t \ge 1$,
$$|EN_{s,t}(1243)| = \frac{1}{(t-1)s + 1} \binom{st}{s}.$$
\end{corollary}
\begin{proof}
This is immediate from Theorem \ref{thm:gen-tree-func} and Corollary \ref{cor:gen-tree-iso}.
\end{proof}
  
The Fuss-Catalan numbers that appear in Corollary \ref{cor:1243 formula} are known to count a variety of generalized Catalan objects. 
For example, it's not difficult to use the discussion following \cite[Eq.~(7.69)]{Graham} to show these numbers count a type of generalized Catalan paths we will call Fuss-Catalan paths.

For all $s \ge 0$ and all $t \ge 2$, a {\em $t$-Fuss-Catalan path of semilength $s$} is a lattice path consisting of $s$ unit East and $(t-1)s$ unit North steps, with the property that each initial string of steps includes at least $t-1$ times as many $N$s as $E$s.
Equivalently, a $t$-Fuss-Catalan path must remain on or above the line ${\displaystyle y=(t-1) x}$.
As we see in Figure \ref{fig:FCpathexample}, 
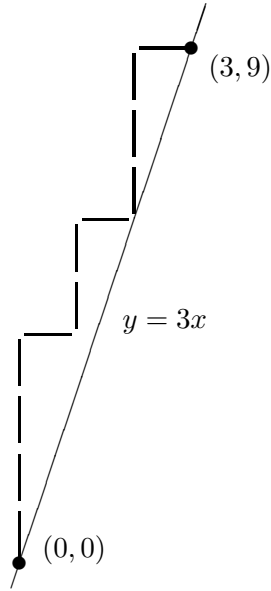
\begin{figure}[ht]
\centering
\setlength{\unitlength}{.3in}
\begin{picture}(3.2,9.5)
\put(-.14,-.43){\line(1,3){3.4}}
\put(0,0){\circle*{.25}}
\put(3,9){\circle*{.25}}
\put(.4,.1){$(0,0)$}
\put(3.3,8.5){$(3,9)$}
\put(1.8,4.1){${\displaystyle y = 3x}$}
\thicklines
\put(0,0.1){\line(0,1){.8}}
\put(0,1.1){\line(0,1){.8}}
\put(0,2.1){\line(0,1){.8}}
\put(0,3.1){\line(0,1){.8}}
\put(0.1,4){\line(1,0){.8}}
\put(1,4.1){\line(0,1){.8}}
\put(1,5.1){\line(0,1){.8}}
\put(1.1,6){\line(1,0){.8}}
\put(2,6.1){\line(0,1){.8}}
\put(2,7.1){\line(0,1){.8}}
\put(2,8.1){\line(0,1){.8}}
\put(2.1,9){\line(1,0){.8}}
\end{picture}
\caption{The $4$-Fuss-Catalan path $NNNNENNENNNE$ and the line $y = 3x$.}
\label{fig:FCpathexample}
\end{figure}
the path $NNNNENNENNNE$ is a $4$-Fuss-Catalan path of semilength $3$.
Note that the 2-Fuss-Catalan paths are the classical Catalan paths.

We conclude this section by giving a constructive bijection between $EN_{s,t}(1243)$ and the set of $t$-Fuss-Catalan paths of semilength $s$, which gives us an alternative proof of Theorem \ref{thm:gen-tree-func} and Corollary \ref{cor:1243 formula}.
As we describe this mapping, we will use the term {\em root} to refer to an element of the $t$th spine.
Thus, the vertices on the bottom-right side of $SAW_{s,t}$ are the roots. 
For example, in Figure \ref{fig:SAW43} the vertices labeled $1$, $4$, $7$, and $10$ are roots. 
In our function, these roots will correspond to unit East steps. 
All other vertices of the poset will correspond to unit North steps.
    
Given a linear extension in $EN_{s,t}(1243)$, we construct the associated $t$-Fuss-Catalan path as follows.
First, in an ordered $s$-tuple record the numbers corresponding to the positions (counted from the start of the linear extension) of the roots. 
Then complement the $s$-tuple with respect to the total number of vertices, $st$. 
Finally, construct a string of $N$s and $E$s with total length $st$ by placing an $E$ in the positions (counted from left to right) corresponding to the integer entries in the complemented $s$-tuple, and then filling in the remaining positions with an $N$.
    
As an example, suppose we are given the linear extension $10\ 11\ 7\ 12\ 8\ 9\ 4\ 5\ 1\ 6\ 2\ 3$ of the poset in Figure \ref{fig:SAW43}.
The $s$-tuple of its root positions is $(1, 3, 7, 9)$, and the complement of this $s$-tuple is $(12, 10, 6, 4)$. 
The resulting lattice path is $NNNENENNNENE$, which is a $3$-Fuss-Catalan path of semilength 4.

In general, the covering relations in $SAW_{s,t}$ guarantee that the $t-1$ non-root vertices in each tooth, which map to North steps, will always precede the root in their tooth, which maps to an East step. 
Consequently, the resulting lattice path will always be a $t$-Fuss-Catalan path of semilength $s$.

To show our map is a bijection, we construct its inverse. 
Given a $t$-Fuss-Catalan path of semilength $s$, read the path from right to left, constructing an $s$-tuple recording the positions (counted from the right end) of the $E$s. 
Map these entries of the $s$-tuple to the positions of the $s$ roots of a linear extension of $EN_{s,t}$. 
Fill these positions with the integers $t(s-1)+1, t(s-2)+1, \dots, t+1, 1$ in decreasing order from left to right. 
Then fill the non-root positions with the integers $t(s-1)+2, t(s-1) +3, \dots, st, t(s-2)+2, t(s-2)+3, \dots, t(s-1), \dots, 2, 3, \dots, t$ in this order from left to right. 
Notice that the relative positioning of the North and East steps required for a $t$-Fuss-Catalan path ensures that at least one root will precede the start of each new block of $t-1$ non-roots in the linear extensions we are constructing, so the resulting permutation is indeed a linear extension of $SAW_{s,t}$. 
The reader can verify that our two functions are inverses of one another, so we have proved the following result.
       
\begin{theorem}
\label{thm:pfaff fuss bijection}
There is a constructive bijection between the set of all linear extensions of $SAW_{s,t}$ and the set of all $t$-Fuss-Catalan paths of semilength $s$.
\end{theorem}

We can modify our construction above to obtain a set of lattice paths which are in bijection with $EN_{s,t}(12354)$.
These paths are another generalization of the classical Catalan paths, but they do not appear to have been previously studied.
We describe these paths in Section \ref{sec:opfd}.

%% file: section5.tex
Now that we have enumerated $EN_{s,t}(1243)$ using both a generating tree and Fuss-Catalan paths, we turn our attention to $EN_{s,t}(2143)$.
As we did for $EN_{s,t}(1243)$, we first show that $EN_{s,t}(2143)$ is actually the set of all linear extensions of another poset.
More specifically, we show that $EN_{s,t}(2143)$ is exactly the set of linear extensions of $EN_{s,t}$ in which all elements of the $j$th tooth appear before any element of the $(j-2)$th tooth.

\begin{theorem}
\label{thm:EN2143sawblade}
For all $s \ge 1$ and all $t \ge 2$, a linear extension of $EN_{s,t}$ avoids $2143$ if and only if $jt$ precedes $(j-3)t+1$ for all $j$ with $3 \le j \le s$.
\end{theorem}
\begin{proof}
($\Rightarrow$)
To prove the contrapositive, suppose $\pi$ is a linear extension of $EN_{s,t}$ such that for some $j$, $3 \le j \le s$, the entry $(j-3)t+1$ precedes $jt$.
Since $(j-3)t+1$ covers $(j-2)t+1$ in $EN_{s,t}$, we know $(j-2)t+1$ precedes $(j-3)t+1$ in $\pi$.
Similarly, since $(j-1)t$ covers $jt$ in $EN_{s,t}$ we know $jt$ precedes $(j-1)t$ in $\pi$.
Now the subsequence $(j-2)t+1,(j-3)t+1,jt,(j-1)t$ of $\pi$ has the same relative order as $2143$, so $\pi$ does not avoid $2143$.

($\Leftarrow$)
To prove the contrapositive in this case, suppose $\pi$ is a linear extension of $EN_{s,t}$ and $a b c d$ is a subsequence of $\pi$ with the same relative order as $2143$.
Since the entries of each tooth appear in increasing order in $\pi$, the entries $a$ and $b$ must be in different teeth in $EN_{s,t}$.
Furthermore, the entries in the higher (that is, lower-numbered) teeth are less than the entries in the lower teeth, so $a$'s tooth is below $b$'s tooth.
For the same reason, $c$'s tooth is also below $a$'s tooth.
Now the first element of $b$'s tooth must appear before the last element of $c$'s tooth, which implies that if $c$ is in tooth $j$ then $(j-3)t+1$ precedes $jt$.
\end{proof}

Recall that a {\em Catalan path} of semilength $n$ is a sequence of $n$ $N$s and $n$ $E$s such that for each $j$, $1 \le j \le 2n$, the number of $N$s among the first $j$ terms of the sequence is greater than or equal to the number of $E$s among the first $j$ terms.
Theorem \ref{thm:EN2143sawblade} allows us to recognize the linear extensions in $EN_{s,t}(2143)$ as a new kind of generalized Catalan path.

\begin{definition}
\label{defn:catzipper}
For any nonnegative integers $s$ and $t$, a {\em Catalan zipper} of dimension $s$ and length $t$ is a sequence of $t$ $N_j$s for each $j$, $1 \le j \le s$, such that the following hold.
\begin{enumerate}
\item
For each $j$, $1 \le j \le s$, if we replace each $N_j$ with $N$, replace each $N_{j+1}$ with $E$, and remove all other entries, then we obtain a Catalan path.
\item
For each $j$, $1 \le j \le s-2$, the rightmost $N_j$ precedes the leftmost $N_{j+2}$.
\end{enumerate}
\end{definition}

Note that the Catalan zippers of dimension $2$ are essentially the Catalan paths.
In addition, the function which replaces each entry $a$ in a linear extension of $EN_{s,t}$ with $N_{t+1-\tooth(a)}$, where $\tooth(a)$ is the number of the tooth containing $a$, is a bijection between $EN_{s,t}(2143)$ and the set of Catalan zippers of dimension $s$ and length $t$.
In Table \ref{table:numberofzippers} 
\begin{table}[ht]
\centering
\begin{tabular}{c|ccccc}
$s \setminus t$ & $1$ & $2$ & $3$ & $4$ & $5$ \\
\hline
$1$ & $1$ & $1$ & $1$ & $1$ & $1$ \\
$2$ & $1$ & $2$ & $5$ & $14$ & $42$ \\
$3$ & $1$ & $4$ & $21$ & $121$ & $728$ \\
$4$ & $1$ & $8$ & $89$ & $1094$ & $14041$ \\
$5$ & $1$ & $16$ & $377$ & $9841$ & $266110$ \\
$6$ & $1$ & $32$ & $1597$ & $88574$ & $5057369$
\end{tabular}
\caption{$|EN_{s,t}(2143)|$ and the number of Catalan zippers of dimension $s$ and length $t$ for small $s$ and $t$.}
\label{table:numberofzippers}
\end{table}
we have the number of linear extensions of $EN_{s.t}(2143)$ and the number of Catalan zippers of dimension $s$ and length $t$ for small $s$ and $t$.

The sequences in the columns of Table \ref{table:numberofzippers} seem to satisfy linear homogeneous recurrence relations with constant coefficients.
To obtain these recurrence relations in general, we consider Catalan zippers according to the number of $N_s$s at their right ends.
More specifically, for positive integers $j$, $s$, and $t$, let $a_{t,j}(s)$ denote the number of Catalan zippers of dimension $s$ and length $t$ whose last $j+1$ entries are $N_{s-1} \underbrace{N_s \cdots N_s}_j$, and note that $|EN_{s,t}(2143)| = \sum_{j=1}^t a_{t,j}(s)$.
We will obtain simple recurrence relations for $a_{t,j}(s)$, but to do this, we will need yet another variation on the idea of a Catalan path.

\begin{definition}
\label{defn:jkcatpath}
For any positive integer $n$ and any integers $j,k$ with $0 \le j \le n$ and $1 \le k \le n$, a {\em $j,k$-Catalan path} $p$ of semilength $n$ is a sequence of $j$ $N$s and $n$ $E$s such that the following hold.
\begin{enumerate}
\item
The sequence $\underbrace{N\cdots N}_{n-j} p$ is a Catalan path.
\item
The last $k+1$ entries of $p$ are $N\underbrace{E\cdots E}_k$, or $k = n$ and $p = \underbrace{E \cdots E}_k$.
\end{enumerate}
We write $b_{j,k}(n)$ to denote the number of $j,k$-Catalan paths of semilength $n$.
\end{definition}

Notice that in any Catalan zipper of dimension $s$ and length $t$, the entries to the right of the rightmost $N_{s-2}$ form a $j,k$-Catalan path of semilength $t$, where $j$ is the number of $N_{s-1}$s to the right of the rightmost $N_{s-2}$ and $k$ is the number of $N_s$s to the right of the rightmost $N_{s-1}$.
This observation allows us to use the numbers $b_{j,k}(n)$ to obtain our recurrence relation for $a_{t,k}(s)$.

\begin{theorem}
\label{thm:atksrecurrence}
For all $s \ge 2$, all $t \ge 1$, and all $k$ with $1 \le k \le t$,
\begin{equation}
\label{eqn:atksrecurrence}
a_{t,k}(s) = \sum_{j=1}^t b_{j,k}(t) a_{t,j}(s-1).
\end{equation}
\end{theorem}
\begin{proof}
The result follows from the fact that each Catalan zipper of dimension $s$ and length $t$ whose last $k+1$ entries are $N_{s-1} \underbrace{N_s \cdots N_s}_k$ can be constructed uniquely as follows.

First choose $j$ with $1 \le j \le t$, and then choose a $j,k$-Catalan path $p$ of semilength $t$.
Now choose a Catalan zipper $\pi'$ of dimension $s-1$ and length $t$ whose last $j+1$ entries are $N_{s-2} \underbrace{N_{s-1} \cdots N_{s-1}}_j$.
To merge $p$ and $\pi'$ into a longer Catalan zipper, first let $p'$ be the sequence we obtain from $p$ by replacing $N$ with $N_{s-1}$ and $E$ with $N_s$.
Then, in $\pi'$, replace the tail $N_{s-2} \underbrace{N_{s-1} \cdots N_{s-1}}_j$ with $p'$.
By construction the resulting sequence is a Catalan zipper of dimension $s$ and length $t$ whose last $k+1$ entries are $N_{s-1} \underbrace{N_s \cdots N_s}_k$.
Furthermore, given such a Catalan zipper, $p'$ is the sequence of entries to the right of its rightmost $N_{s-2}$ and $\pi'$ is the Catalan zipper we obtain by removing all of the $N_s$s from $\pi$.
Therefore, each such Catalan zipper is uniquely constructed as described, and the result follows.
\end{proof}

Equation \eqref{eqn:atksrecurrence} is only useful if we can compute $b_{j,k}(n)$ efficiently.
Fortunately, as we show next, $b_{j,k}(n)$ satisfies a remarkably simple recurrence relation.

\begin{theorem}
\label{thm:brecurrence}
We have $b_{1,1}(1) = 1$, and for all $n \ge 2$ and all $j$ and $k$ with $1 \le j,k \le n$,
\begin{equation}
\label{eqn:brecurrence}
b_{j,k}(n) = b_{j,k}(n-1) + b_{j-1,k}(n).
\end{equation}
\end{theorem}
\begin{proof}
The fact that $b_{1,1}(1) = 1$ is clear.
In addition, it's not difficult to check that $b_{1,k}(n) = 1$ for all $k$ with $1 \le k \le n$, and that $b_{j,n}(n) = 1$ for all $j$ with $1 \le j \le n$, so \eqref{eqn:brecurrence} holds for $j = 1$ and $k = n$.

When $j > 1$ and $k < n$ there are two kinds of $j,k$-Catalan paths:  those beginning with $E$ and those beginning with $N$.
Removing the leading $E$ from a path of the first kind results in a $j,k$-Catalan path of semilength $n-1$.
Conversely, if $p$ is a $j,k$-Catalan path of semilength $n-1$ then $E p$ is a $j,k$-Catalan path of semilength $n$ which begins with $E$.
Similarly, removing the leading $N$ from a path of the second kind results in a $j-1,k$-Catalan path of semilength $n$, and if $p$ is a $j-1,k$-Catalan path of semilength $n$ then $N p$ is a $j,k$-Catalan path of semilength $n$ which begins with $N$.
Now the result follows.
\end{proof}

If we write $B_n$ to denote the $n \times n$ matrix whose $j,k$th entry is $b_{j,k}(n)$ then we have
$$B_2 = \left(\begin{matrix}1 & 1 \\ 1 & 1 \end{matrix}\right),$$
$$B_3 = \left( \begin{matrix} 1 & 1 & 1 \\ 2 & 2 & 1 \\ 2 & 2 & 1 \end{matrix} \right),$$
and
$$B_4 = \left(\begin{matrix} 1 & 1 & 1 & 1 \\ 3 & 3 & 2 & 1 \\ 5 & 5 & 3 & 1 \\ 5 & 5 & 3 & 1 \end{matrix}\right).$$
As these examples suggest, $B_n$ has a symmetry which is not apparent in \eqref{eqn:brecurrence}.
To prove this symmetry persists, we first prove an analogue of \eqref{eqn:brecurrence}.

\begin{theorem}
For all $n \ge 2$ and all $j$ and $k$ with $1 \le j,k \le n$,
\begin{equation}
\label{eqn:otherbrecurrence}
b_{j,k}(n) = b_{j,k+1}(n) + b_{j-1,k-1}(n-1).
\end{equation}
\end{theorem}
\begin{proof}
In the proof of Theorem \ref{thm:brecurrence} we saw that $b_{1,k}(n) = 1$ for all $k$ with $1 \le k \le n$ and $b_{j,n}(n) = 1$ for all $j$ with $1 \le j \le n$, so \eqref{eqn:otherbrecurrence} holds for $j = 1$ and $k = n$.
When $j > 1$ and $k < n$ there are two kinds of $j,k$-Catalan paths:  those in which the entry immediately preceding the rightmost $N$ is $E$ and those in which it is $N$.
In the first case, moving the $E$ immediately preceding the rightmost $N$ to the right end of the path is a bijection with the set of $j,k+1$-Catalan paths of semilength $n$.
In the second case, removing the rightmost $N$ and the rightmost $E$ is a bijection with the set of $j-1,k-1$-Catalan paths of semilength $n-1$.
\end{proof}

Combining \eqref{eqn:brecurrence} and \eqref{eqn:otherbrecurrence} leads directly to a proof of our observed symmetry.

\begin{theorem}
For all $n \ge 1$ and all $j,k$ with $1 \le j,k \le n$, we have $b_{j,k}(n) = b_{n+1-k,n+1-j}(n)$.
That is, the number of $j,k$-Catalan paths of semilength $n$ is equal to the number of $n+1-k,n+1-j$-Catalan paths of semilength $n$.
\end{theorem}
\begin{proof}
The result is clear when $n+j \le 2$, so we argue by induction on $n+j$.
Using \eqref{eqn:brecurrence} and our induction hypothesis we find
\begin{eqnarray*}
b_{j,k}(n) &=& b_{j,k}(n-1) + b_{j-1,k}(n) \\
&=& b_{n-k,n-j}(n-1)+b_{n+1-k,n+2-j}(n).
\end{eqnarray*}
Now the result follows from \eqref{eqn:otherbrecurrence} when we replace $j$ with $n+1-k$ and $k$ with $n+1-j$.
\end{proof}

As our small examples suggest, we can also show that the upper right corner of $B_n$ is the first $n$ rows of Pascal's triangle.

\begin{theorem}
For all $n \ge 1$ and all $j$ and $k$ with $k \ge j$ we have $b_{j,k}(n) = \binom{n-k+j-1}{j-1}$.
\end{theorem}
\begin{proof}
We saw in the proof of Theorem \ref{thm:brecurrence} that $b_{1,k}(n) = 1$ for all $k$ with $1 \le k \le n$ and $b_{j,n}(n) = 1$ for all $j$ with $1 \le j \le n$, so the result holds for $j = 1$ and $k = n$.
Arguing by induction on $j+n$ and using \eqref{eqn:brecurrence}, we find
\begin{eqnarray*}
b_{j,k}(n) &=& b_{j,k}(n-1) + b_{j-1,k}(n) \\
&=& \binom{n-1-k+j-1}{j-1}+\binom{n-k+j-2}{j-2} \\
&=& \binom{n-k+j-1}{j-1},
\end{eqnarray*}
as desired.
\end{proof}

We can now use the transfer matrix method and standard results about rational generating functions to find $|EN_{s,t}(2143)|$ for small values of $t$.

\begin{theorem}
\label{thm:2143formulas}
For all $s \ge 1$,
\begin{enumerate}[{\upshape (i)}]
\item
$|EN_{s,1}(2143)| = 1$;
\item
$|EN_{s,2}(2143)| = 2^{s-1}$;
\item
$|EN_{s,3}(2143)| = F_{3s-1}$;
\item
${\displaystyle |EN_{s,4}(2143)| = \frac12 \left(3\cdot 9^{s-1} + (-1)^s\right)}$.
\end{enumerate}
Here $F_n$ is the $n$th Fibonacci number, defined with $F_0 = 0$ and $F_1 = 1$.
\end{theorem}
\begin{proof}
(i)
$EN_{s,1}$ has only one linear extension, namely $s, s-1, \ldots,1$, and this linear extension avoids $2143$.

(ii)
By \cite[Thm.~4.7.2]{StanleyVol1} and \cite[Thm.~4.1.1]{StanleyVol1}, since $\det(I_2 - x B_2) = 1-2x$, the sequence $\{|EN_{s,2}(2143)|\}_{s=1}^\infty$ satisfies the recurrence relation $a_n = 2a_{n-1}$.
Now (ii) follows by induction on $s$.

(iii)
This is similar to the proof of (ii).
Specifically, since $\det(I_3-xB_3) = 1-4x-x^2$, the sequence $\{|EN_{s,3}(2143)|\}_{s=1}^\infty$ satisfies the recurrence relation $a_n = 4a_{n-1}+a_{n-2}$.
Now the result follows by induction on $s$, since
\begin{eqnarray*}
4 F_{3s-4}+F_{3s-7} &=& 4 F_{3s-4}+F_{3s-5}-F_{3s-6} \\
&=& F_{3s-3}+3F_{3s-4}-F_{3s-6} \\
&=& F_{3s-3}+2F_{3s-4}+F_{3s-5} \\
&=& 2F_{3s-3}+F_{3s-4} \\
&=& F_{3s-2}+F_{3s-3} \\
&=& F_{3s-1}.
\end{eqnarray*}

(iv)
This is similar to the proofs of (ii) and (iii), using the fact that $\det(I_4-xB_4) = 1-8x-9x^2$.
\end{proof}

We can use the same techniques as in the proof of Theorem \ref{thm:2143formulas} to find a recurrence relation satisfied by $\{|EN_{s,t}(2143)|\}_{s=1}^\infty$ for any $t \ge 1$.
For example, when $t = 5$ we find the sequence satisfies $a_n = 16 a_{n-1}+3a_{n-2}-235a_{n-3}+36a_{n-4}$.
But when $t \ge 5$ this does not seem to result in a simple formula for $|EN_{s,2}(2143)|$.

%% file: section6.tex
In this section we extend Yakoubov's work in a new direction, turning our attention from classical enumeration of pattern avoiding linear extensions to $q$-enumeration of these linear extensions.
More specifically, recall that an {\em inversion} in a permutation (or, more generally, any finite sequence of integers) $\pi$ is a pair $i < j$ such that $\pi(i) > \pi(j)$;  we write $\inv(\pi)$ to denote the number of inversions in $\pi$.
For any rectangular poset $P$ and any permutation $\sigma$, let $P(\sigma)(q)$ be the generating function given by
$$P(\sigma)(q) = \sum_{\pi \in P(\sigma)} q^{\inv(\pi)}.$$
In this section we determine $P(\sigma)(q)$ for a variety of $P$ and $\sigma$.

Several of our results in this section are most easily stated in terms of well-known $q$-analogues of the Catalan numbers.
To describe these $q$-analogues, recall that a {\em Catalan word} of length $2n$ is a sequence of $n$ $0$s and $n$ $1$s in which the number of $0$s in each initial segment is greater than or equal to the number of $1$s in that initial segment.
Following \cite{fh}, we write $C_n(q)$ to denote the polynomial given by
\begin{equation}
\label{eqn:Cnq}
C_n(q) = \sum_{w \in CW_n} q^{\inv(w)},
\end{equation}
where $CW_n$ is the set of Catalan words of length $2n$.
Sometimes it is more convenient to work with the polynomial we get by reversing the coefficients of $C_n(q)$, so we again follow \cite{fh} by writing $\tilde{C}_n(q)$ to denote the polynomial given by
\begin{equation}
\label{eqn:Cntildeq}
\tilde{C}_n(q) = q^{\binom{n}{2}} C_n\left(q^{-1}\right).
\end{equation}
Although we will not need them, it is worth noting that these polynomials satisfy the recurrences
\begin{equation}
\label{eqn:Crecurrence}
C_n(q) = \sum_{k=1}^n q^{k(n-k)} C_{k-1}(q) C_{n-k}(q)
\end{equation}
and
\begin{equation}
\label{eqn:Ctilderecurrence}\tilde{C}_n(q) = \sum_{k=1}^n q^{k-1} \tilde{C}_{k-1}(q) \tilde{C}_{n-k}(q)
\end{equation}
for $n \ge 1$.
We can now prove $q$-analogues of some of our results in Section \ref{sec:lengththree}.

\begin{theorem}
\label{thm:qCats}
For all $s \ge 1$ and all $t \ge 1$,
\begin{enumerate} [{\upshape (i)}] 
\item
$EN_{2,t}(321)(q) = q^{\binom{t+1}{2}} \tilde{C}_t(q);$
\item
$EN_{s,2}(123)(q) = q^{3\binom{s}{2}} C_s(q);$
\item
$NE_{s,2}(123)(q) = q^{\frac{s(3s-1)}{2}} C_s(q);$
\item
$NE_{2,t}(123)(q) = q^{\frac{t(3t-1)}{2}} \tilde{C}_t(q).$
\end{enumerate}
\end{theorem}
\begin{proof}
To prove (i), first associate with each linear extension $\pi$ in $EN_{2,t}(321)$ a sequence $w(\pi)$ by replacing, in $\pi^r$ (the reverse of $\pi$), each of $1,2,\ldots,t$ with 0 and each of $t+1,t+2,\ldots,2t$ with $1$.
Note that this gives a bijection between $EN_{2,t}(321)$ and the set of Catalan words of length $2t$.
Furthermore, each inversion in $w(\pi)$ corresponds to an inversion in $\pi^r$.
The entries $1,2,\ldots,t$ and the entries $t+1,t+2,\ldots,2t$ appear in decreasing order in $\pi^r$, and these are the only other inversions in $\pi^r$, so by \eqref{eqn:Cnq} we have
$$\sum_{\pi \in EN_{2,t}(321)} q^{\inv(\pi^r)} = q^{2\binom{t}{2}}C_t(q).$$
Since $\inv(\pi^r) = \binom{2t}{2}-\inv(\pi)$, equation \eqref{eqn:Cntildeq} and some algebra complete the proof of (i).

To prove (ii), in each $\pi \in EN_{s,2}(123)$ we consider four kinds of inversions $i < j$, according to the parities of $\pi(i)$ and $\pi(j)$.
There are $2\binom{s}{2}$ of these in which $\pi(i)$ and $\pi(j)$ have the same parity, and there are $\sum_{j=1}^s (j-1) = \binom{s}{2}$ of them when $\pi(i)$ is odd and $\pi(j)$ is even.
To handle the case in which $\pi(i)$ is even and $\pi(j)$ is odd, note that we can associate with each $\pi$ a Catalan word $w(\pi)$ by replacing each odd entry with $0$ and each even entry with $1$, and this correspondence is a bijection between $EN_{s,2}(123)$ and the set of Catalan words of length $2s$.
Furthermore, $i<j$ is an inversion in $w(\pi)$ if and only if it is an inversion in $\pi$ in which $\pi(i)$ is even and $\pi(j)$ is odd.
Combining these observations with \eqref{eqn:Cnq} completes the proof of (ii).

The proof of (iii) is similar to the proof of (ii):  if $\pi \in NE_{s,2}(123)$ then $\pi$ has $2 \binom{s}{2}$ inversions $i < j$ in which $\pi(i)$ and $\pi(j)$ have the same parity, $\sum_{j=1}^s j = \binom{s+1}{2}$ inversions $i < j$ in which $\pi(i)$ is even and $\pi(j)$ is odd, and $C_s(q)$ is the generating function for the inversions in which $\pi(i)$ is odd and $\pi(j)$ is even.

Finally, the proof of (iv) is similar to the proof of (i).
\end{proof}

Note that the pattern avoidance conditions in Theorem \ref{thm:qCats} can be dropped, since in each case every linear extension of the given poset avoids the given pattern.
As a result, Theorem \ref{thm:qCats} also gives us $q$-analogues of Corollary \ref{cor:linextCat}.

Our last $q$-analogue of a result in Section \ref{sec:lengththree} involves the classical $q$-integers, defined by $[n]_q = 1+q+\cdots+q^{n-1}$ for $n \ge 1$.

\begin{theorem}
\label{thm:stq}
For all $s \ge 1$ and all $t \ge 1$,
\begin{equation}
\label{eqn:NEst213q}
NE_{s,t}(213)(q) = q^{s\binom{t}{2} +t\binom{s}{2}+\frac{(s-1)(t-1)(st-2)}{2}} [t]_q^{s-1}.
\end{equation}
\end{theorem}
\begin{proof}
The only linear extension in $NE_{1,t}(213)$ is $t (t-1) \cdots 2 1$, so $NE_{1,t}(213)(q) = q^{\binom{t}{2}}$, and the result holds for $s = 1$.
Now suppose $s > 1$ and \eqref{eqn:NEst213q} holds with $s$ replaced by $s-1$;  we argue by induction on $s$.

In the proof of Theorem \ref{thm:NE(213)=NE(132)} we showed that $\pi \in NE_{s,t}(213)$ if and only if there is a $j$ with $1 \le j \le t$ and a linear extension $\pi' \in NE_{s-1,t}(213)$ such that $\pi$ has the form
$$\pi = \underbrace{st, (st-1), \ldots, st-j+1}_j,\underbrace{(s-1)t}_{\pi'(1)}, \underbrace{st-j,\ldots, (s-1)t-1}_{t-j},\pi'(2), \pi'(3), \ldots, \pi'((s-1)t).$$
Now the inversions in $\pi$ consist of the inversions in $\pi'$, the inversions among the $t$ largest entries of $\pi$, the inversions in which $\pi'(1) = (s-1)t$ is the smaller entry, and the inversions involving one of the $t$ largest entries of $\pi$ and one of the entries $\pi'(k)$ for $2 \le k \le (s-1)t$.
This means
\begin{equation}
\label{eqn:invrecurrence}
\inv(\pi) = \inv(\pi') + \binom{t}{2} + j + t((s-1)t-1).
\end{equation}
Using \eqref{eqn:invrecurrence}, we find
$$NE_{s,t}(213)(q) = \sum_{\pi' \in NE_{s-1,t}(213)} \sum_{j=1}^t q^{\binom{t}{2}+j+t((s-1)t-1)} q^{\inv(\pi')}.$$
Now by induction we have
\begin{eqnarray*}
NE_{s,t}(213)(q) &=& q^{\binom{t}{2}+t ((s-1)t-1)} \left(\sum_{\pi' \in NE_{s-1,t}(213)} q^{\inv(\pi')}\right) [t]_q \\
&=& q^{s\binom{t}{2} +t\binom{s}{2}+\frac{(s-1)(t-1)(st-2)}{2}} [t]_q^{s-1},
\end{eqnarray*}
as desired.
\end{proof}

Theorems \ref{thm:qCats} and \ref{thm:stq} include $q$-analogues of all of our results for one forbidden pattern of length three, so we now turn our attention to forbidden patterns of length four.
In these cases we do not have explicit factorizations of $EN_{s,t}(\sigma)(q)$ or $NE_{s,t}(\sigma)(q)$ for any $\sigma$, but we do have a result concerning the maximum and minimum degrees of the terms in $EN_{s,t}(1243)(q)$.

\begin{theorem}
\label{thm:invmaxmin}
For any positive integers $s$ and $t$, the inversion numbers of the linear extensions of $EN_{s,t}$ which avoid 1243 have minimum $(t^2 - t + 1) \binom{s}{2}$ and maximum $t^2 \binom{s}{2}$.
\end{theorem}
\begin{proof}
To show the minimum number of inversions is $(t^2 - t + 1) \binom{s}{2}$, first note that by Theorem 5.1, each entry which is not on the $t$th tooth must form an inversion with all entries on higher (that is, lower-numbered) spines. 
In particular, each entry on the $j$th tooth that is not on the $t$th spine contributes $(j-1)(t-1)$ inversions, so the $j$th tooth contributes $(j-1)(t-1)^2$ inversions. 
Now note that each entry on the $n$th tooth forms an inversion with all of the entries on higher (that is, lower-numbered) teeth, so the $j$th entry on the $t$th tooth contributes $(j-1)t$ inversions. 
In total, then, each linear extension of $EN_{s,t}$ which avoids 1243 has at least $\sum_{j=1}^s ((j-1)(t-1)^2 + (j-1)t) = \binom{s}{2} (t^2 - t + 1)$ inversions.
Since the linear extension which begins with the entries on the $s$th spine has no other inversions, it must have exactly $\binom{s}{2} (t^2 - t + 1)$ inversions.
    
To show the maximum number of inversions is $t^2 \binom{s}{2}$, for each linear extension $\pi$ we consider pairs $i < j$ with $\pi(i) < \pi(j)$.
That is, we consider the noninversions or coinversions in $\pi$.
On each tooth there are $\binom{t}{2}$ such pairs, so no linear extension has more than $\binom{st}{2} - s \binom{t}{2} = t^2 \binom{s}{2}$ inversions.
Since the linear extension we obtain by traversing the teeth from $s$th to first has no other noninversions, it must have exactly $t^2 \binom{s}{2}$ inversions.
\end{proof}

Theorem \ref{thm:invmaxmin} captures a striking observation about the data in Table \ref{table:invmaxmin},
\begin{table}[ht]
\centering
\begin{tabular}{|c|c|c|c|c|c|c|c|} 
 \hline
 $s \setminus t$ & 1 & 2 & 3 & 4 & 5 & 6 \\ 
 \hline
 1 & 0 & 0 & 0 & 0 & 0 & 0 \\ 
 2 & 1 & 3-4 & 7-9 & 13-16 & 21-25 & 31-36 \\
 3 & 3 & 9-12 & 21-27 & 39-48 & 63-75 & 93-108 \\
 4 & 6 & 18-24 & 42-54 & 78-96 & 126-150 & 196-216 \\
 5 & 10 & 30-40 & 70-90 & 130-160 & 210-250 & 310-360 \\
 6 & 15 & 45-60 & 105-135 & 195-240 & 315-375 & 465-540 \\
 \hline
\end{tabular}
\caption{The range of inversion numbers (minimum-maximum) for linear extensions of $EN_{s,t}(1243)$.}
\label{table:invmaxmin}
\end{table}
which shows the minimum and maximum inversion numbers among linear extensions in $EN_{s,t}(1243)$ for various $s$ and $t$.
Namely, for linear extensions in $EN_{s,t}(1243)$, the minimum inversion number is the product of the inversion number of the linear extension in $EN_{s,1}(1243)$ and the minimum inversion number over linear extensions in $EN_{2,t}(1243)$. 
Similarly, the maximum inversion number is the product of the inversion number of the linear extension in $EN_{s,1}(1243)$ and the maximum inversion number over linear extensions in $EN_{2,t}(1243)$.
In other words, in Table \ref{table:invmaxmin} the entry in row $i$ and column $j$ is the product of the entry in row $i$ and column 1 with the entry in row $2$ and column $j$.

%% file: section7.tex
The first and most fundamental open problem related to this work is to enumerate $NE_{s,t}(123)$ when $s \ge 3$ and $t \ge 3$.
See Table \ref{table:NEst123data} for the values of $|NE_{s,t}(123)|$ for some small $s$ and $t$.
Yakoubov \cite{combs} also has some open problems involving monotone forbidden patterns of length three, so it would also be interesting to connect $NE_{s,t}(123)$ with one or more of these problems.
Similarly, Levin, Pudwell, Riehl, and Sandberg \cite{heaps} have some open problems involving binary heaps avoiding monotone patterns, and it would be interesting to connect $NE_{s,t}(123)$ with one or more of these problems.

At the end of Section \ref{sec:bn} we noted that the set of linear extensions of $EN_{s,t}$ (or $NE_{s,t}$) is naturally in bijection with the set of standard tableaux of shape $t^s$, which we can enumerate with the classical hook length formula.
Later we showed that the set of linear extensions of $EN_{s,t}$ which avoid $1243$ is in bijection with the set of linear extensions of the sawblade poset $SAW_{s,t}$.
By rotating the Hasse diagram for $SAW_{s,t}$ through $3\pi/4$ radians clockwise and enclosing each vertex in a box, we see that the linear extensions of $SAW_{s,t}$ are in bijection with the ``standard tableaux'' of a new partition-type shape.
Figure \ref{fig:sawtableaux}
\begin{figure}[ht]
\centering
\setlength{\unitlength}{.2in}
\begin{picture}(10,3)
\multiput(0,0)(1,0){2}{\line(0,1){3}}
\put(2,2){\line(0,1){1}}
\multiput(3,1)(1,0){2}{\line(0,1){2}}
\put(5,1){\line(0,1){1}}
\multiput(6,0)(1,0){2}{\line(0,1){2}}
\multiput(8,0)(1,0){3}{\line(0,1){1}}
\put(0,3){\line(1,0){4}}
\put(0,2){\line(1,0){7}}
\put(0,1){\line(1,0){1}}
\put(3,1){\line(1,0){7}}
\put(0,0){\line(1,0){1}}
\put(6,0){\line(1,0){4}}
\end{picture}
\caption{A ``sawblade partition'' shape for $s = 3$ and $t = 4$.}
\label{fig:sawtableaux}
\end{figure}
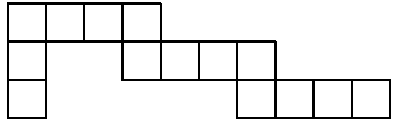
shows the ``sawblade partition'' shape for $s = 3$ and $t = 4$.
In view of this connection between linear extensions and this type of generalized standard tableaux, and because the number of linear extensions of $SAW_{s,t}$ is given by a simple ratio of products, we believe there is a hook length formula explaining these results.
However, we have not yet found such a formula.

At the end of Section \ref{sec:gentrees} we mentioned that $EN_{s,t}(12354)$ is in bijection with another set of generalized Catalan paths.
These paths are the sequences of $s$ $N_1$s, $s$ $N_2$s, and $(t-2)s$ $E$s such that the subsequence of $N_1$s and $N_2$s is a Catalan path (with $N_2$ playing the role of $E$), while the subsequence of $N_2$s and $E$s is a $t-1$-Fuss-Catalan path of semilength $s$.
These paths do not appear to have been previously studied, so it might be fruitful to explore their enumerative and combinatorial properties.
In addition, their form may suggest other interesting generalized Catalan paths.

In Section \ref{sec:statistics} we studied $P(\sigma)(q)$, the generating function with respect to inversion number for linear extensions of a rectangular poset $P$ which avoid $\sigma$, but there is much more to do in this direction.
For example, in most of our Section \ref{sec:statistics} results we had $|\sigma| = 3$, so it's natural to ask whether there are similar results for any $\sigma$ with $|\sigma| \ge 4$.
As a first step toward answering this question, we have found that computer generated data support the following conjecture, which we have verified for $t \le 9$.

\begin{conjecture}
For all $t \ge 1$, 
$$EN_{3,2t-1}(1243)(q) = q^{3(t^2-t+1)} [2t-1]_q [4t-1]_q.$$
\end{conjecture}

More generally, we conjecture that there is a natural statistic $\istat$ on the set $FS_{s,t}$ of $t$-Fuss-Catalan paths of semilength $s$ such that the $q$-Fuss-Catalan number given by
$$\sum_{p \in FC_{s,t}} q^{\istat(p)}$$
satisfies $q$-analogues of \eqref{eqn:Crecurrence} and \cite[Eq.~(7.68)]{Graham}, and that $EN_{s,t}(1243)(q) = q^{f(s,t)} C(q)$ for an appropriate function $f(s,t)$.
In addition, we expect $\istat$ to be closely related to $\inv$.

Turning from $1243$ to $2143$, computer generated data also support our next conjecture, which we have verified for $s \le 10$.
This conjecture is a $q$-analogue of Theorem \ref{thm:2143formulas}(ii).

\begin{conjecture}
\label{conj:2qinv}
For all $s \ge 1$,
$$EN_{s,2}(2143)(q) = q^{(2s-1)(s-1)} (1+q)^{s-1}.$$
\end{conjecture}

\noindent
Our next conjecture, which we have verified for $s \le 9$, is a $q$-analogue of Theorem \ref{thm:2143formulas}(iii).

\begin{conjecture}
\label{conj:Fqinv}
For all $s \ge 1$,
$$EN_{s,3}(2143)(q) = q^{9\binom{s}{2}} F_s\left(\frac{1}{q}\right),$$
where $F_s(q)$ is defined by $F_0(q) = 1$, $F_1(q) = 1$, and $F_s(q) = (1+q+2q^2)F_{s-1}(q) + q^3 F_{s-2}(q)$ for $s \ge 2$.
\end{conjecture}

The polynomials $F_s(q)$ have a variety of interesting properties in addition to the fact that $F_s(1)$ is the Fibonacci number $F_{3s-1}$.
For example, when $s \ge 2$ it's not difficult to prove by induction that $F_s(q)$ has degree $2s-1$, the coefficient of $q^{2s-1}$ is $2^{s-2}$, the constant term is $F_s(0) = 1$, the coefficient of $q$ is $s-1$, and the coefficient of $q^2$ is $\frac{s(s+3)}{2}$.
In addition, we make the following conjectures, all of which we have verified for $s \le 10$.

\begin{conjecture}
For all $s \ge 2$, the coefficient of $q^3$ in $F_s(q)$ is the binomial transform of the sequence obtained by interleaving $n+1$ and $2n+1$.
This is OEIS sequence A098156.
\end{conjecture}

\begin{conjecture}
For all $s \ge 2$, the coefficient of $q^{2s-2}$ in $F_s(q)$ the number of compositions of $s+9$ into $s$ parts, none of which is 2 or 3.
This is OEIS sequence A134465.
\end{conjecture}

\begin{conjecture}
For all $s \ge 2$, the coefficient of $q^{s+2}$ in $F_s(q)$ is the number of jumps in all binary trees with $s$ edges.
This is OEIS sequence A127531.
\end{conjecture}

\begin{conjecture}
For all $s \ge 2$, the coefficient of $q^s$ is given by OEIS sequence A072547.
\end{conjecture}

\begin{conjecture}
For all $s \ge 2$, the coefficient of $q^{s+1}$ in $F_s(q)$ is $\binom{2s+1}{s-1}$.
\end{conjecture}

\begin{conjecture}
For all $s \ge 2$, the coefficient of $q^{s-1}$ in $F_s(q)$ is the number of hill-free Dyck paths of semilength $s$.
This is OEIS sequence A116914.
\end{conjecture}

\begin{conjecture}
For all $s \ge 0$, the sequence of coefficients of $F_s(q)$ is unimodal.
\end{conjecture}

Finally, since $F_s(q)$ is a $q$-analogue of $F_{3s-1}$, it would be interesting and useful to have a $q$-analogue $f_n(q)$ of the Fibonacci numbers such that $F_s(q)$ is $f_{3s-1}(q)$, up to an appropriate power of $q$.

Of course, $\inv$ is just one of many useful and interesting statistics on permutations, so it's natural to ask whether there are results similar to those we have in Section \ref{sec:statistics} for other permutation statistics.
For example, recall that a {\em descent} $i$ in a permutation (or any sequence of integers) $\pi$ is a position with $\pi(i) > \pi(i+1)$, and the {\em major index} of $\pi$, written $\maj(\pi)$, is the sum of the descents in $\pi$.
Then for any rectangular poset $P$ and any forbidden pattern $\sigma$, we define $P(\sigma)[q]$ by
\begin{equation}
P(\sigma)[q] = \sum_{\pi \in P(\sigma)} q^{\maj(\pi)}.
\end{equation}

Although we have no results involving $P(\sigma)[q]$, we do have several conjectures.
To state these conjectures, recall that F\"urlinger and Hofbauer \cite{fh} and Andrews \cite{andrews} (among others) have studied another $q$-Catalan number, which can be defined in terms of the major index.
In particular, let $c_n(q)$ denote the polynomial given by
$$c_n(q) = \sum_{\pi \in CW_n} q^{\maj(\pi)},$$
where the sum on the right is over all Catalan words of length $2n$.
With this notation we have the following conjecture, which we have verified for $s \le 9$ and $t \le 9$.

\begin{conjecture}
For all $s \ge 1$ and all $t \ge 1$,
\begin{enumerate} [{\upshape (i)}] 
\item
$EN_{2,t}(321)[q] = q^t c_t(q);$
\item
$EN_{s,2}(123)[q] = q^{2\binom{s}{2}} c_s(q);$
\item
$NE_{s,2}(123)[q] = q^{s^2} c_s(q);$
\item
$NE_{2,t}(123)[q] = q^{t^2} c_t(q).$
\end{enumerate}
\end{conjecture}

There are even more interesting directions to proceed in the study of $EN_{s,t}(\sigma)$ when $|\sigma| = 4$.
For example, we can show that when $\sigma = 1243$ the maximum major index of any linear extension in $EN_{s,t}(1243)$ is twice the minimum major index of any such linear extension (a proof we leave as an exercise for the reader), but we have no such results for other forbidden patterns $\sigma$, and we know little else about the distribution of $\maj$ even on $EN_{s,t}(1243)$.
Along these lines, however, it appears that descents in Fuss-Catalan words correspond under our bijection to descents in linear extensions of $EN_{s,t}$ which avoid $1243$, though the converse is not true.
It would be interesting to find a statistic on Fuss-Catalan paths which corresponds with major index on linear extensions under our bijection.
More generally, many statistics on permutations and lattice paths have been studied in various contexts, so it seems likely there are correspondences between statistics on these objects waiting to be discovered.